\documentclass[12pt,a4paper]{scrartcl}

\usepackage{amsmath, amssymb, amsfonts}
\usepackage{amsthm}
\usepackage{mathtools}
\usepackage{graphics}
\usepackage{epstopdf}
\usepackage{natbib}
\usepackage{bbm}
\usepackage{setspace}

\usepackage{bbm}		
\usepackage{color}
\usepackage{booktabs}

\newtheorem{theorem}{Theorem}

\newtheorem{corollary}[theorem]{Corollary}

\newtheorem{proposition}[theorem]{Proposition}

\theoremstyle{definition}

\newtheorem{example}[theorem]{Example}
\newtheorem{definition}[theorem]{Definition}
\newtheorem{remark}[theorem]{Remark}

\DeclareMathOperator*{\esssup}{ess\,sup}
\DeclareMathOperator*{\essinf}{ess\,inf}

\newcommand{\VK}{\stackrel{\mathcal{D}}{\longrightarrow}}
\newcommand{\FSK}{\stackrel{a.s.}{\longrightarrow}}

\begin{document}
\bibliographystyle{natbib}

\title{Stochastic orders and  measures of skewness and dispersion based on expectiles}
\author{Andreas Eberl\footnote{andreas.eberl@kit.edu},
Bernhard Klar\footnote{ bernhard.klar@kit.edu} \\
\small{Institute of Stochastics,} \\
\small{Karlsruhe Institute of Technology (KIT), Germany.}
}


\date{\today }
\maketitle

\begin{abstract}
Recently, expectile-based measures of skewness akin to well-known quantile-based skewness measures have been introduced, and it has been shown that these measures possess quite promising properties \citep{ek-exl,ek}. However, it remained unanswered whether they preserve the convex transformation order of van Zwet, which is sometimes seen as a basic requirement for a measure of skewness.
It is one of the aims of the present work to answer this question in the affirmative.
These measures of skewness are scaled using interexpectile distances.
We introduce orders of variability based on these quantities and show that the so-called weak expectile dispersive order is equivalent to the dilation order. Further, we analyze the statistical properties of empirical interexpectile ranges in some detail.
\end{abstract}

{\bfseries Keywords:} Expectile, skewness, stop-loss transform, dispersion order, dilation order, dispersive order, scale measure, asymptotic relativ efficiency.

\section{Introduction}

Over the last years, there was a steady increase in literature dealing with expectiles. These are measures of non-central location that have properties similar to quantiles. Therefore, expectiles can also be used as building blocks for measures of scale, skewness, etc. However, no such measure should be used  without identification of the ordering it preserves.

Let us explain this in more detail. Since the seminal work of \cite{zwet}, \cite{oja}, \cite{macgillivray}, an  axiomatic approach to measure statistical quantities is commonly accepted. It involves two main steps \citep{huerlimann}:
\begin{itemize}
\item
Define stochastic (partial) orders on sets of random variables or distribution functions
that allow for comparisons of the given statistical quantity.
\item
Identify measures of the statistical quantity by considering functionals of distributions that preserve the partial order, and use only such measures in practical work.
\end{itemize}

Given a dispersion or variability ordering $\leq_D$,
the general axiomatic approach requires that a scale measure $\delta$ satisfies
\begin{enumerate}
\item[D1.]
For $c,d\in \mathbb{R}$, $\delta(cX+d)=\lvert c \rvert \delta(X)$.
\item[D2.]
If $X\leq_D Y$, then $\delta(X)\leq \delta(Y)$.
\end{enumerate}
Similarly, given a skewness order $\leq_S$, a skewness measure $\gamma$ should satisfy
\begin{enumerate}
\item[S1.]
For $c>0$ and $d\in \mathbb{R}$, $\gamma(cX+d) = \gamma(X)$.
\item[S2.]
The measure satisfies $\gamma(-X)=-\gamma(X)$.
\item[S3.]
If $X\leq_S Y$, then $\gamma(X)\leq \gamma(Y)$.
\end{enumerate}

The generally accepted strongest dispersion and skewness orders are the dispersive order \citep{bickel4} and the convex transformation order \citep{zwet}, respectively.

Research which treats expectile-based measures and related stochastic orderings includes \cite{bellini,bkmr,bell-et-al,km,ek,ek-exl,arab}. In particular, \cite{ek-exl,ek} introduced expectile-based measures of skewness which possess quite promising properties and have close connections to other skewness functionals. However, it remained unanswered whether these measures preserve the convex transformation order.
It is one of the aims of the present work to answer this question in the affirmative.

As part of these measures of skewness, interexpectile distances (also called interexpectile ranges) appear quite naturally; they have also been used in a finance context by \cite{bellini2018b,bellini2020,bellini2021}. We introduce orders of variability based on slightly more general quantities and show that the so-called weak expectile dispersive order is equivalent to the dilation order. Hence, interexpectile ranges preserve the dispersive ordering.

Up to now, statistical properties of empirical interexpectile ranges do not seem to have been investigated. We show that they are a good compromise between the standard deviation on the one hand and robust measures as the interquartile range on the other.

\smallskip
This paper is organized as follows. In Section \ref{sec-expectile}, we recall the definitions of expectiles and some of their properties.
Section \ref{sec-exp-skew} discusses expectile skewness. It is shown that expectile-based skewness measures are consistent with the convex transformation order. In Section \ref{sec-exp-disp}, we introduce strong and weak expectile dispersive orders and show that the latter is equivalent to the dilation order. It follows that the interexpectile range preserves the dispersive order. These concepts are illustrated in Section \ref{sec-lomax}, using the Lomax or Pareto type II distribution.
Empirical interexpectile ranges are analyzed and compared with other scale measures in Section \ref{sec-exp-scale}, using standardized asymptotic relative efficiencies.

\section{Preliminaries} \label{sec-expectile}

Throughout the paper, we assume that all mentioned random variables $X$ are non-degenerate, have a finite mean (denoted as $X \in L^1$), and are defined on a common probability space $(\Omega,\mathcal{A}, P)$.
Further, we assume that the supports of the underlying distributions are intervals and that these distributions have no atoms. Hence, the distribution function (cdf) $F_X$ of $X$ has a strictly increasing and continuous inverse on $(0,1)$.

Expectiles $e_X(\alpha)$ of a random variable $X\in L^{2}$ have been defined by \citet{newey} as the minimizers of an asymmetric quadratic loss:%
\begin{equation}
e_X(\alpha)=\arg \min_{t\in \mathbb{R}}\left\{ E \ell_\alpha(X-t) \right\}
\text{,}  \label{exp_def_1}
\end{equation}%
where
\begin{equation*}
\ell_\alpha(x) =
\begin{cases}
\alpha x^2, & \mbox{ if } x \ge 0, \\
(1-\alpha) x^2, & \mbox{ if } x < 0,%
\end{cases}%
\end{equation*}
and $\alpha \in (0,1)$. For $X\in L^{1}$ (but $X\notin L^{2}$), equation (\ref{exp_def_1}) has to
be modified \citep{newey} to
\begin{equation}
e_X(\alpha) = \arg \min_{t\in \mathbb{R}} \left\{ E\left[ \ell_\alpha(X-t) -
\ell_\alpha(X) \right]\right\}.  \label{exp_def_3}
\end{equation}
The minimizer in (\ref{exp_def_1}) or (\ref{exp_def_3}) is always unique and
is identified by the first order condition%
\begin{equation}
\alpha E\left( X-e_{X}(\alpha )\right) _{+}=(1-\alpha )E\left(
X-e_{X}(\alpha )\right) _{-}\text{,}  \label{exp_def_2}
\end{equation}%
where $x_{+}=\max \{x,0\}$, $x_{-}=\max \{-x,0\}$.
This is equivalent to characterizing expectiles via an identification function, which, for any $\alpha \in (0, 1)$ is defined by
\begin{equation*}
	I_\alpha(x, y) = \alpha (y - x) \mathbbm{1}_{\{y \geq x\}} - (1 - \alpha) (x - y) \mathbbm{1}_{\{y < x\}}
\end{equation*}
for $x, y \in \mathbb{R}$. The $\alpha$-expectile of a random variable $X \in L^1$ is then the unique solution of
\begin{equation*}
	E I_\alpha(t, X) = 0, \quad t \in \mathbb{R}.
\end{equation*}
Similarly, the {\itshape empirical $\alpha$-expectile} $\hat{e}_n(\alpha)$ of a sample $X_1,\ldots,X_n$ is defined as solution of
\begin{equation*}
	I_\alpha(t, \hat{F}_n) = \frac{1}{n} \sum_{i=1}^{n} I_\alpha(t, X_i) = 0, \quad t \in \mathbb{R}.
\end{equation*}
Like quantiles, expectiles are measures of non-central location and have similar properties (see, e.g., \citet{newey} and \citet{bkmr}).
Clearly, expectiles depend only on the distribution of the random variable
$X$ and can be seen as statistical functionals defined on the set of distribution functions with finite mean on $\mathbb{R}$.

Throughout the paper, we make use of the following stochastic orders:
\begin{definition}
\begin{enumerate}
\item[(i)]
$F_X$ precedes $F_Y$ in the usual stochastic order (denoted by $F_X \leq_{st} F_Y$) if $E\phi(X)\le E\phi(Y)$ for all increasing functions $\phi$, for which both expectations exist.
\item[(ii)]
$F_X$ precedes $F_Y$ in the convex order (denoted by $F_X \leq_{cx} F_Y$) if $E\phi(X) \le E\phi(Y)$ for all convex functions $\phi$, for which both expectations exist.
\item[(iii)]
$F_X$ precedes $F_Y$ in the convex transformation order (denoted by $F_X\leq _{c} F_Y$)
if $F_Y^{-1}\circ F_X$ is convex.
\item[(iv)]
$F_X$ precedes $F_Y$ in the expectile (location) order (denoted by $F_X\leq _{e} F_Y$)
if
\begin{align*}
e_X(\alpha) \leq e_Y(\alpha) \mbox{ for all } \alpha \in (0,1).
\end{align*}
\end{enumerate}
\end{definition}

The first three orders are well-known \citep{shaked,mueller-stoyan}; the expectile (location) order was introduced in \cite{bell-et-al}.
Since the usual stochastic order $\leq_{st}$ is equivalent to the pointwise ordering of the quantiles, the definition of the expectile ordering is quite natural, just replacing quantiles by expectiles.

As noted by \citet{jones}, expectiles are the quantiles of a suitably transformed distribution. Indeed, the first order condition (\ref{exp_def_2}) can be written in the equivalent form
\begin{align} \label{exp-cdf}
\alpha &=\frac{E\left( X-t\right) _{-}}{E\left\vert X-t\right\vert}
=\frac{\pi_X(t)-\mu+t}{2\pi_X(t)-\mu+t}
=:\breve{F}_X(t),
\end{align}
where $\mu=EX$ and $\pi_X(t) = E(X-t)_+$  denotes the  stop-loss transform of $X$. Hence, $\breve{F}_X(t)=e_X^{-1}(t)$, and the expectile order could also be defined by
$\breve{F}_X(t)\geq \breve{F}_Y(t)$ for all $t$.
The expectile order is weaker than the usual stochastic order, i.e. $X\leq _{st}Y$ implies $X\leq _{e}Y$ \citep{bellini}.

\section{Expectile skewness} \label{sec-exp-skew}

In the following, we summarize some results which are important with respect to expectile-based quantification of skewness. \cite{bell-et-al}, Thm.\ 12 and Cor.\ 13, show the following equivalence.

\begin{proposition} \label{prop1}
Let $X,Y\in L^{1}$ with $EX=EY=\mu $, $m=\essinf(X), M=\esssup(X)$. Then the following conditions are equivalent:
\begin{enumerate}
\item[a)]
$X\leq _{e}Y$.
\item[b)]
$\pi _{X}(x)\geq \pi _{Y}(x)$, for each $x\in (m,\mu )$ and
$\pi _{X}(x)\leq \pi _{Y}(x)$, for each $x\in (\mu ,M)$.
\item[c)]
$(X-\mu )_{-}\geq _{cx}(Y-\mu )_{-}$ and
$(X-\mu )_{+}\leq _{cx}(Y-\mu )_{+}$.
\end{enumerate}
\end{proposition}

\cite{ek-exl} introduced a family of scalar measures of skewness
\begin{equation*} 
\tilde{s}_2(\alpha) = \frac{ e_X(1-\alpha)+e_X(\alpha)-2\mu }{ e_X(1-\alpha)-e_X(\alpha) }, \quad  \alpha\in(0,1/2),
\end{equation*}
and called a distribution right-skewed (left-skewed) in the expectile sense if
$\tilde{s}_2(\alpha)\geq (\leq)\, 0$ for all $\alpha\in(0,1/2)$ and equality does not hold for each $\alpha\in(0,1/2)$. If equality holds for each $\alpha$, the distribution is symmetric.
They also defined a normalized version $s_2(\alpha)=\tilde{s}_2(\alpha)/(1-2\alpha)$, and proved that $-1<s_2(\alpha)<1$, and both inequalities are sharp for any $\alpha \in (0, 1/2)$.

Moreover, \cite{ek-exl} introduced a function
\begin{equation*}
S_X(t) = \frac{\pi_X(\mu+t)-\pi_X(\mu-t)}{t} + 1
= \frac{1}{t} \int_{\mu-t}^{\mu + t} F_X(z) dz - 1, \quad t>0,
\end{equation*}
which has been shown to be strongly related to $s_2$:
it holds that $s_2(\alpha)\geq 0$ for $\alpha\in (0,1/2)$ if and only if $S_X(t)\geq 0$ for each $t>0$ \citep{ek-exl}. Based on $S_X$, the following  two skewness orders have been defined, which are both weaker than van Zwet's skewness order $\leq _c$ \citep{ek-exl}.

\begin{definition}
\begin{enumerate}
\item[a)]
Let $\tilde{S}_X(t)=S_X(td_X)$, where $d_X=E\lvert X-EX\rvert$ denotes the mean absolute deviation from the mean (MAD). Then, we write $F_X\leq_{sf} F_Y$ if
$\tilde{S}_X(t) \leq \tilde{S}_Y(t) \, \forall \, t>0$.
\item[b)]
$F_Y$ is more skew with respect to mean and MAD than $F_X$ ($F_X <_\mu^{d} F_Y$), if $F_X(d_X x+EX)$ and $F_Y(d_Y x+EY)$ cross each other exactly once on each side of $x=0$, with $F_X(EX)\leq F_Y(EY)$.
\end{enumerate}
\end{definition}

\smallskip
\cite{arab} introduced a skewness order as follows.

\begin{definition} \label{def-s}
Let $X,Y\in L^{1}$, and define $\tilde{X}=(X-EX)/E\lvert X-EX\rvert, \tilde{Y}=(Y-EY)/E\lvert Y-EY\rvert$, with cdf's $F_{\tilde{X}}$ and $F_{\tilde{Y}}$.
$X$ is smaller than $Y$ in the $s$-order, denoted by $F_X\leq_s F_Y$, if
\begin{align*}
\int_{-\infty}^x F_{\tilde{X}}(t)dt &\geq \int_{-\infty}^x F_{\tilde{Y}}(t)dt, \quad \forall x \leq 0,
\end{align*}
and
\begin{align*}
\int_x^{\infty} \bar{F}_{\tilde{X}}(t)dt &\leq
\int_x^{\infty} \bar{F}_{\tilde{Y}}(t)dt, \quad \forall x\geq 0,
\end{align*}
where $\bar{F}=1-F$ denotes the survival function.
\end{definition}

Using well-known conditions for the convex order (see, e.g., (3.A.7) and (3.A.8) in \cite{shaked}), $X\leq_s Y$ holds if $\tilde{X}_{-}\geq _{cx}\tilde{Y}_{-}$ and
$\tilde{X}_{+}\leq _{cx}\tilde{Y}_{+}$. A comparison with Prop. \ref{prop1}c) yields that
\begin{align} \label{e-s-equiv}
X\leq_s Y  \quad \Leftrightarrow \quad   \tilde{X} \leq _{e} \tilde{Y}
\end{align}
With regard to the standardized variables, (\ref{e-s-equiv}) establishes that larger expectiles correspond to a more skewed distribution.
The equivalence in (\ref{e-s-equiv}) was derived in \cite{arab}, Thm. 14, using a different argument.

Definition \ref{def-s} entails $\int_{-x}^{x} (F_{\tilde{X}}(t)-F_{\tilde{Y}}(t)) dt \leq  0$  for all $x>0$.
However, this is equivalent to  $\tilde{S}_X(x) \leq \tilde{S}_Y(x)$ for all $x>0$. Hence, $\leq_{sf}$, the skewness order defined by $\tilde{S}_X$, is weaker than the  $\leq_s$-order. On the other hand, the proof of Thm.\ 13(2.) in \cite{arab} shows that
$\leq_s$ is weaker than the order $\leq_\mu^d$. Taking into account Thm.\ 3 in \cite{ek-exl}, we obtain the following chain of implications:
\begin{align} \label{order-imp}
X\leq _c Y \;\Rightarrow\; X\leq_\mu^d Y \;\Rightarrow\;
X\leq_s Y \;\Rightarrow\; X\leq_{sf} Y.
\end{align}

The following theorem is the main result in this section, showing that the expectile-based skewness measure $s_2$ (and, hence, $\tilde{s}_2$) preserves van Zwet's skewness order.

\begin{theorem} \label{s2-cons}
Let $X,Y\in L^{1}$. Then, $s_2$ is consistent with $\leq _c$, i.e.
$X\leq _c Y$ implies $s_{2,X}(\alpha)\leq s_{2,Y}(\alpha)$ for each $\alpha\in (0,1/2)$.
\end{theorem}

\begin{proof}
By (\ref{order-imp}) and (\ref{e-s-equiv}), $X\leq _c Y$ implies $\tilde{X} \leq _{e} \tilde{Y}$, i.e. $e_{\tilde{X}}(\alpha) \leq e_{\tilde{Y}}(\alpha)$ for all $\alpha\in (0,1)$. A straightforward computation shows that
$s_{2,\tilde{X}}(\alpha)\leq s_{2,\tilde{Y}}(\alpha)$ is equivalent to
\begin{align*}
e_{\tilde{Y}}(1-\alpha) e_{\tilde{X}}(\alpha) \leq
e_{\tilde{X}}(1-\alpha) e_{\tilde{Y}}(\alpha),
\quad \forall \alpha\in(0,1/2),
\end{align*}
which, in turn, is equivalent to
\begin{align} \label{exp-ord-equiv}
e_{\tilde{Y}}(1-\alpha)\, \lvert e_{\tilde{X}}(\alpha)\rvert \geq
e_{\tilde{X}}(1-\alpha)\,  \lvert e_{\tilde{Y}}(\alpha)\rvert,
\quad \forall \alpha\in(0,1/2).
\end{align}
Since, by assumption, $e_{\tilde{Y}}(1-\alpha)\geq e_{\tilde{X}}(1-\alpha)$ and
$\lvert e_{\tilde{X}}(\alpha)\rvert
\geq \lvert e_{\tilde{Y}}(\alpha)\rvert$ for $\alpha\in(0,1/2)$,
inequality (\ref{exp-ord-equiv}) holds.
Since $s_2$ is invariant under location-scale transforms, $s_{2,X}(\alpha)\leq s_{2,Y}(\alpha)$ holds as well for all $\alpha\in(0,1/2)$.
\end{proof}

\begin{remark} \label{order-s2}
If we write $X \leq_{s_2} Y$ if $s_{2,X}(\alpha)\leq s_{2,Y}(\alpha)$ for all $\alpha\in(0,1/2)$, the proof of Thm. \ref{s2-cons} shows that $\leq_{s_2}$ is a weaker order of skewness than $\leq_{s}$.
The question if $\leq_{s_2}$ implies $\leq_{sf}$ or vice versa is open.
Another expectile based order, stronger than $\leq_{s_2}$, could be defined analogously to the quantile-based skewness order $\leq_c$, i.e. by
\begin{align} \label{expskew}
\frac{e_X(w)-2e_X(v)+e_X(u)}{ e_X(w)-e_X(u)} &\leq
\frac{e_Y(w)-2e_Y(v)+e_Y(u)}{ e_Y(w)-e_Y(u)}
\end{align}
for all $0<u<v<w<1$; this order is equivalent to the convexity of $\breve{F}_Y^{-1} \circ \breve{F}_X$, where $\breve{F}_X$ is defined in (\ref{exp-cdf}) (cp. \cite{ek-exl} for the quantile case). However, the above results are not strong enough to show that this order is weaker than $\leq_c$.
\end{remark}

For the suitability of $s_2$ as a skewness measure, including its empirical counterpart, we refer to \cite{ek-exl} and \cite{ek}. To gain an impression, consider a Bernoulli distribution with success probability $p$. Clearly, for $p=1/2$, any skewness measure should be zero. For decreasing $p$, the distribution becomes more and more skewed to the right (note that the center decreases to 0), and a skewness measure should converge to its maximal possible value. Similarly, it should converge to its minimal possible value for $p\to 1$. For the expectile skewness, we obtain $s_2(\alpha)=1-2p$, independent of $\alpha$. The moment skewness is $(1-2p)/(p(1-p))$; quantile-based skewness measures are not uniquely defined for discrete distributions.

\section{Expectile dispersive order and related measures} \label{sec-exp-disp}

Additionally to the stochastic orders in Section \ref{sec-expectile}, we need the following dispersion or variability orders.

\begin{definition} \label{disp-order}
\begin{enumerate}
\item[(i)]
$F_X$ precedes $F_Y$ in the dispersive order
(written $F_X \leq_{disp} F_Y$) if
\begin{align} \label{disp-ord}
F_X^{-1}(v)-F_X^{-1}(u) \leq F_Y^{-1}(v)-F_Y^{-1}(u)  \quad \forall \ 0<u<v<1.
\end{align}
\item[(ii)]
$F_X$ precedes $F_Y$ in the weak dispersive order
(written $F_X \leq_{w-disp} F_Y$) if (\ref{disp-ord}) is fulfilled for all
$0<u \leq 1/2 \leq v<1$.
\item[(iii)]
$F_X$ precedes $F_Y$ in the expectile dispersive order (written $F_X \leq_{e-disp} F_Y$) if
\begin{align} \label{edisp-ord}
e_X(v)-e_X(u) \leq e_Y(v)-e_Y(u)  \quad \forall \ 0<u<v<1.
\end{align}
\item[(iv)]
$F_X$ precedes $F_Y$ in the weak expectile dispersive order (written $F_X \leq_{we-disp} F_Y$) if (\ref{edisp-ord}) is fulfilled for all $0<u \leq 1/2 \leq v<1$.
\item[(v)]
$F_X$ precedes $F_Y$  in the dilation order (written $F_X \leq_{dil} F_Y$) if
\begin{align} \label{dil-ord}
X-EX \leq_{cx} Y-EY.
\end{align}
\end{enumerate}
\end{definition}

\begin{remark}
The first ordering in Def. \ref{disp-order} is well-known (see, e.g., \cite{bickel4}, \cite{oja}, \cite{shaked82}, \cite{shaked}). The defining condition of the weak dispersive order, which is obviously weaker than the dispersive order, can be equivalently written as
\begin{align} \label{w-disp-order}
F_Y^{-1}(u)-F_X^{-1}(u) &\leq (\geq) \, F_Y^{-1}(1/2)-F_X^{-1}(1/2)
\quad \text{for } u \leq (\geq ) 1/2.
\end{align}
This ordering was used in \cite{bickel3} for symmetric distribution; for arbitrary distributions, it was introduced in \cite{macgillivray}, Def.\ 2.6, and denoted by $\leq_1^m$. In introducing strong and weak expectile dispersive orders, we mimic these definitions. The following results show that the weak ordering defined in such a way has appealing theoretical properties; further, it corresponds to the expectile-based scale measures treated in Sec. \ref{sec-exp-scale}. The expectile equivalent of the (strong) dispersive order from Def.\ \ref{disp-order}(i)
is shortly discussed in Example \ref{example-se-disp} and Remark \ref{lomax-remark}.

We remark that \cite{bellini2021} have introduced the following expectile based variability order which is even weaker than the weak expectile dispersive order:
\[
X \le_{\Delta-ex} Y,  \quad \text{if }
e_X(1-\alpha)-e_X(\alpha) \leq e_Y(1-\alpha)-e_Y(\alpha)
\; \forall \ \alpha \in (0,1/2).
\]
This order can be seen as a dispersion counterpart to the skewness order $\leq_{s_2}$ in Remark \ref{order-s2}.
The dilation order in Def. \ref{disp-order}(v) is less well-known compared to the dispersive order, see \cite{belzunce-et-al} or \cite{fagiuoli}. However, many of its properties follow from properties of the convex order.
\end{remark}

Our main result in this section shows that the weak expectile dispersive order and the dilation order coincide. The proof uses the idea in the proof of Thm.\ 14 in \cite{arab}.

\begin{theorem} \label{theo-dil}
Let $X,Y\in L^1$ with strictly increasing cdf's $F_X$ and $F_Y$ which are  continuous on their supports. Then, $F_X \leq_{we-disp} F_Y$ if, and only if, $F_X \leq_{dil} F_Y$.
\end{theorem}

\begin{proof}
Define $\tilde{X}=X-EX, \tilde{Y}=Y-EY$. Then, $X \leq_{dil} Y$ is equivalent to
$\tilde{X} \leq_{cx} \tilde{Y}$, and therefore to
$\pi_{\tilde{X}}(t)\leq \pi_{\tilde{Y}}(t) \, \forall t$.
By continuity of $\pi_{\tilde{X}}$ and $\pi_{\tilde{Y}}$, this is equivalent to
$\pi_{\tilde{X}}(t)/\lvert t\rvert \leq \pi_{\tilde{Y}}(t)/\lvert t\rvert  \, \forall t\neq 0$, and hence to
\begin{align} \label{dil-equiv1}
\frac{\pi_{\tilde{X}}(t)}{t} \geq \frac{\pi_{\tilde{Y}}(t)}{t} \; \forall t<0 \quad \text{and} \quad
\frac{\pi_{\tilde{X}}(t)}{t} \leq \frac{\pi_{\tilde{Y}}(t)}{t} \; \forall t>0.
\end{align}
Note that, using the properties of the stop-loss transform \citep[Thm.\ 1.5.10]{mueller-stoyan}, $\pi_{\tilde{X}}(t)/t \leq -1$ for $t<0$, and $\pi_{\tilde{X}}(t)/t \geq 0$ for $t>0$.
Now, applying to both sides of the inequalities the transformation
$h(x)=(x+1)/(2x+1), x\neq -1/2$, which is decreasing for $x<-1/2$ as well as for $x>-1/2$, shows that (\ref{dil-equiv1}) is equivalent to
\begin{align} \label{dil-equiv2}
\breve{F}_{\tilde{X}}(t) \leq \breve{F}_{\tilde{Y}}(t) \; \forall t<0 \quad \text{and} \quad
\breve{F}_{\tilde{X}}(t) \geq \breve{F}_{\tilde{Y}}(t) \; \forall t>0,
\end{align}
where $\breve{F}_X$ is the expectile cdf defined in (\ref{exp-cdf}).
In turn, (\ref{dil-equiv2}) is equivalent to
\begin{align*}
e_{\tilde{X}}(p) \geq e_{\tilde{Y}}(p) \; \forall p<1/2 \quad \text{and} \quad
e_{\tilde{X}}(p) \leq e_{\tilde{Y}}(p) \; \forall p>1/2.
\end{align*}
This means that $e_X(p)-EX \geq e_Y(p)-EY \, \forall p<0 \quad \text{and} \quad
e_X(p)-EX \leq e_Y(p)-EY \, \forall p>0$, which is equivalent to
$X \leq_{we-disp} Y$ (cp. the representation of the weak dispersive order in (\ref{w-disp-order})).
\end{proof}

\begin{remark}
\cite{bellini}, Thm.\ 3(b) already proved that $X\leq_{cx} Y$ implies
$e_X(\alpha) \geq e_Y(\alpha)$ for each $\alpha\leq 1/2$ and
$e_X(\alpha) \leq e_Y(\alpha)$ for each $\alpha\geq 1/2$.
From this, the implication $F_X \leq_{dil} F_Y \,\Rightarrow \, F_X \leq_{we-disp} F_Y$ follows, see also \cite{bellini2018b}. However, Thm. \ref{theo-dil} also yields the reverse direction. Moreover, its proof is rather elementary,
whereas \cite{bellini} employed techniques from monotone comparative statics.
\end{remark}

Since, for random variables with finite mean, the dispersive order implies the dilation order (\cite{shaked}, Thm.\ 3.B.16), the next result is a direct consequence of Thm.\ \ref{theo-dil}.

\begin{corollary}
Let $F_X$ and $F_Y$ be strictly increasing and continuous on their supports with finite expectations. Then, $X\leq _{disp}Y$ implies $X\leq_{we-disp}Y$.
\end{corollary}

In general, the (strong) expectile dispersive order in Def. \ref{disp-order}(iii) is rather difficult to handle. However, the following example shows that this order does not imply the dispersive order (see also Remark \ref{lomax-remark}c)).

It is an open question if the reverse direction holds, i.e. if the dispersive order implies the expectile dispersive order.

\begin{example} \label{example-se-disp}
	Let $p_X, p_Y \in (0, 1)$, $p_X \neq p_Y$ and $0 < a_X < a_Y$. Furthermore, let $\tilde{X} \sim \mathrm{Bin}(1, p_X), \tilde{Y} \sim \mathrm{Bin}(1, p_Y)$ and $X = a_X \cdot \tilde{X}$ and $Y = a_Y \cdot \tilde{Y}$. It follows directly that $X$ and $Y$ are not comparable with respect to $\leq_{disp}$ since $\mathrm{range}(F_X) \subseteq \mathrm{range}(F_Y)$ is a necessary condition for $X \leq_{disp} Y$ \citep[Thm.\ 1.7.3]{mueller-stoyan}. Further, a simple calculation yields $\pi_X(t) = p_X (a_X - t)$ for $t \in [0, a_X]$. It follows that
	\begin{equation*}
		\breve{F}_X(t) = \frac{t(1 - p_X)}{p_X a_X + t (1 - 2p_X)} \quad \text{and} \quad e_X(\alpha) = \frac{\alpha p_X a_X}{(1-\alpha) + p_X(2\alpha - 1)}
	\end{equation*}
	for $t \in [0, a_X]$ and $\alpha \in (0, 1)$ with analogous results for $Y$. Overall,
	\begin{equation*}
		(e_Y \circ \breve{F}_X)(t) = \frac{p_Y (1- p_X) a_Y t}{p_X (1 - p_Y) a_X + t (p_Y - p_X)}
	\end{equation*}
	for $t \in [0, a_X]$. Since $X \leq_{disp} Y$ is equivalent to $(F_Y^{-1} \circ F_X)' \geq 1$ \citep[see, e.g.,][p.\ 157]{oja}, $X \leq_{e-disp} Y$ is equivalent to $(e_Y \circ \breve{F}_X)' \geq 1$. Because of $\lim_{p_X \to p_Y} (e_Y \circ \breve{F}_X)(t) = \frac{a_Y}{a_X} t$ for all $t \in [0, a_X]$, $X \leq_{e-disp} Y$ holds if the difference between $p_X$ and $p_Y$ is sufficiently small.\par
	Since the involved distributions are discrete, this example does not fit in the general setting of this work. However, the statement of this example remains valid if the distributions of both $X$ and $Y$ are sufficiently closely approximated by continuous distributions (e.g.\ by linear interpolation). Overall, it is proved that $X \leq_{e-disp} Y \not\Rightarrow X \leq_{disp} Y$ in general.
\end{example}

\section{Dispersion orders for the Lomax distribution} \label{sec-lomax}
To illustrate the various concepts, we consider the Lomax or Pareto type II distribution having density, distribution and quantile function
\begin{align*}
f(t) &= f(t;\alpha,\lambda) \ = \ \frac{\alpha}{\lambda} \left(1+\frac{t}{\lambda}\right)^{-(\alpha+1)}, \quad t\geq 0, \\
F(t) &= F(t;\alpha,\lambda) \ = \ 1 - \left(1+t/\lambda\right)^{-\alpha}, \quad t\geq 0, \\
q(p) &= q(p;\alpha,\lambda)
\ = \ \lambda \left( (1-p)^{-1/\alpha}-1 \right), \quad 0<p<1,
\end{align*}
where $\alpha$ and $\lambda$ are positive parameters. Accordingly, the stop-loss transform is given by
\begin{align*}
\pi(t) &= \frac{\lambda}{\alpha-1} \left(1+\frac{t}{\lambda}\right)^{1-\alpha}, \quad t\geq 0,
\end{align*}
and the expectile cdf $\breve{F}$ can explicitly computed by (\ref{exp-cdf}).
Further, if $X\sim F(\cdot;\alpha,\lambda)$,
\begin{align*}
 EX=\frac{\lambda}{\alpha-1} \; (\alpha>1), & \quad
 Var(X)= \frac{\lambda^2\alpha}{(\alpha-1)^2(\alpha-2)} \; (\alpha>2).
\end{align*}
In the following, assume $X\sim F(t;\alpha_1,\lambda_1)$ and $Y\sim F(t;\alpha_2,\lambda_2)$.

\subsection{Dispersive order}
If $X$ and $Y$ have densities $f_X$ and $f_Y$, then $X\le_{disp}Y$ if, and only if,
\begin{align*}
  f_Y\left( F^{-1}_Y(p) \right) & \leq f_X\left( F^{-1}_X(p) \right) \quad \forall p\in(0,1)
\end{align*}
\citep[(3.B.119)]{shaked}. Applied to the Lomax distribution, $X\le_{disp}Y$ iff
\begin{align} \label{cond-disp}
\frac{\alpha_2}{\lambda_2}(1-p)^{1/\alpha_2} &\leq
\frac{\alpha_1}{\lambda_1}(1-p)^{1/\alpha_1}, \quad 0<p<1.
\end{align}
For $p$ converging to 0, (\ref{cond-disp}) is fulfilled if
\begin{equation} \label{cond-disp1}
\frac{\lambda_1}{\alpha_1} \leq \frac{\lambda_2}{\alpha_2}.
\end{equation}
Looking at $p\to 1$ shows that
\begin{equation} \label{cond-disp2}
\alpha_1 \geq \alpha_2
\end{equation}
is a second necessary condition for $X\le_{disp}Y$.
However, (\ref{cond-disp1}) and (\ref{cond-disp2}) are also sufficient:
Taking the reciprocal of (\ref{cond-disp}) shows that the inequality is equivalent to
\begin{equation} \label{cond-disp3}
\frac{1}{\alpha_1} q(p;\alpha_1,\lambda_1) + \frac{\lambda_1}{\alpha_1}
\; \leq \;
\frac{1}{\alpha_2} q(p;\alpha_2,\lambda_2)+\frac{\lambda_2}{\alpha_2}.
\end{equation}
Now, (\ref{cond-disp1}) and (\ref{cond-disp2}) are necessary and sufficient for $X\leq_{st} Y$ \citep{bell-et-al}. Hence,
$q(p;\alpha_1,\lambda_1)\leq q(p;\alpha_2,\lambda_2)$ for all $p$, and (\ref{cond-disp3}) holds under (\ref{cond-disp1}) and (\ref{cond-disp2}).
Overall, we have
\begin{align*}
  X\leq_{st} Y  \text{ as well as } X\leq_{disp} Y  & \;
  \text{ iff \; (\ref{cond-disp1})  and  (\ref{cond-disp2}) hold}.
\end{align*}

\subsection{Weak expectile dispersive order}
Here, we have to assume $\alpha_1,\alpha_2>1$ for the expected values to exist. Define $H(p)=\frac{1}{1-p}\int_p^1 F^{-1}(u)du$. \cite{fagiuoli} showed that $X\leq_{dil}Y$ if, and only if,
\begin{align} \label{cond-dil1}
   H_X(p)-H_X(0)& \leq H_Y(p)-H_Y(0), \quad  0<p<1.
\end{align}
For the Lomax distribution, $H(p)=(\alpha q(p)+\lambda)/(\alpha-1)$ and
$H(0)=\lambda/(\alpha-1)$. Hence, (\ref{cond-dil1}) holds iff
\begin{align} \label{cond-dil2}
 \frac{\alpha_1 \, q(p;\alpha_1,\lambda_1)}{\alpha_1-1} &\leq
 \frac{\alpha_2 \, q(p;\alpha_2,\lambda_2)}{\alpha_2-1}, \quad  0<p<1.
\end{align}
A discussion of the behavior of $q(p)$ for $p\to 1$ shows that (\ref{cond-disp2}) is necessary for (\ref{cond-dil2}).
Further, a second order Taylor expansion around $p=0$ yields
$q(p)=\lambda p/\alpha+O(p^2)$. Therefore, (\ref{cond-dil2}) can only be satisfied if
\begin{equation} \label{cond-dil}
\frac{\lambda_1}{\alpha_1-1} \leq \frac{\lambda_2}{\alpha_2-1}.
\end{equation}
Thus, (\ref{cond-disp2}) and (\ref{cond-dil}) are necessary for $X\leq_{dil}Y$.
We now show that they are also sufficient.
Since $cq_X(p)=q_{cX}(p)$, and $\lambda$ is a scale parameter, (\ref{cond-dil2}) is equivalent to
\begin{align*}
 q\left(p;\alpha_1,\frac{\alpha_1\lambda_1}{\alpha_1-1}\right)  &\leq
 q\left(p;\alpha_2,\frac{\alpha_2\lambda_2}{\alpha_2-1}\right), \quad  0<p<1.
\end{align*}
This, in turn, is equivalent to
\begin{align} \label{cond-dil3}
 \bar{F}(p;\alpha_1,\tilde{\lambda}_1)  &\leq
 \bar{F}(p;\alpha_2,\tilde{\lambda}_2), \quad  0<p<1,
\end{align}
where $\tilde{\lambda}_i=\alpha_i\lambda_i/(\alpha_i-1), i=1,2.$ The above results about the usual stochastic order yield that (\ref{cond-dil3}) is satisfied iff
$\alpha_1\geq \alpha_2$ and
$\tilde{\lambda}_1/\alpha_1 \leq \tilde{\lambda}_2/\alpha_2$, which coincide with
(\ref{cond-disp2}) and (\ref{cond-dil}).

\cite{bell-et-al} have shown that these two conditions are also necessary and sufficient for the so-called increasing convex order ($\leq_{icx}$). Hence, we have the following result:
\begin{align*}
  X\leq_{we-disp} Y  \text{ as well as } X\leq_{icx} Y  & \;
  \text{ iff \; (\ref{cond-disp2})  and  (\ref{cond-dil}) hold}.
\end{align*}

\begin{remark} \label{lomax-remark}
\begin{enumerate}
\item[a)]
Assume $\alpha_1\geq \alpha_2>2$  and that (\ref{cond-dil}) holds. Then,
$\alpha_1/(\alpha_1-2)\leq \alpha_2/(\alpha_2-2)$, and $Var(X)\leq Var(Y)$ follows. Generally, the variance preserves the weak expectile dispersive order and therefore also the dilation order (see also the following section).
\item[b)]
If $\alpha_1\geq \alpha_2$ and $\lambda_1/(\alpha_1-1)=\lambda_2/(\alpha_2-1)$, i.e. $EX=EY$, then $X\leq_{icx} Y$ corresponds to $X\leq_{cx} Y$. According to the last section, $X\leq_{we-disp} Y$ holds as well in this case.

On the other hand, if $X$ and $Y$ are random variables from different Lomax
distributions, but with equal means, they can never be ordered in expectile order \citep{bell-et-al}. Since the dispersive order implies the stochastic order for distributions with the same finite left endpoint of their supports \citep[Thm.\ 3.B.13]{shaked}, the expectile dispersive order then also implies the expectile (location) order (by applying the cited result to the expectile cdf's $\breve{F}_X$ and $\breve{F}_Y$). Hence, $X$ and $Y$ can also not be ordered with respect to the expectile dispersive ordering. This example shows that $\leq_{we-disp}$ is strictly weaker than $\leq_{e-disp}$.
\item[c)]
Let $\alpha_1 = 3, \lambda_1 = \sqrt{3}, \alpha_2 = 2, \lambda_2 = 1$. Since $\alpha_1 >\alpha_2$ and $EX=\sqrt{3}/2 < EY=1$, one has $X \leq_{we-disp} Y$. On the other hand, (\ref{cond-disp1}) is not satisfied. Hence, $X \le_{disp} Y$ does not hold.
In this example, $X$ also precedes $Y$ in the expectile location order \citep[Thm.\ 23]{bell-et-al}. The left panel of Figure \ref{fig:lomax} shows the interexpectile ranges of $X$ and $Y$; clearly, $e_X(1-p)-e_X(p)\leq e_Y(1-p)-e_Y(p)$ for $0<p<1$.
The right panel shows a plot of $e_Y(\breve{F}_X(x))-x$, which is increasing in $x$. This indicates that also $X \leq_{e-disp} Y$ holds \citep[(3.B.10)]{shaked}. Hence, similarly to Example \ref{example-se-disp}, this shows that $X \leq_{e-disp} Y$ does not imply $X \le_{disp} Y$ in general.

\begin{figure}
\centering{\includegraphics[width=0.95\textwidth]{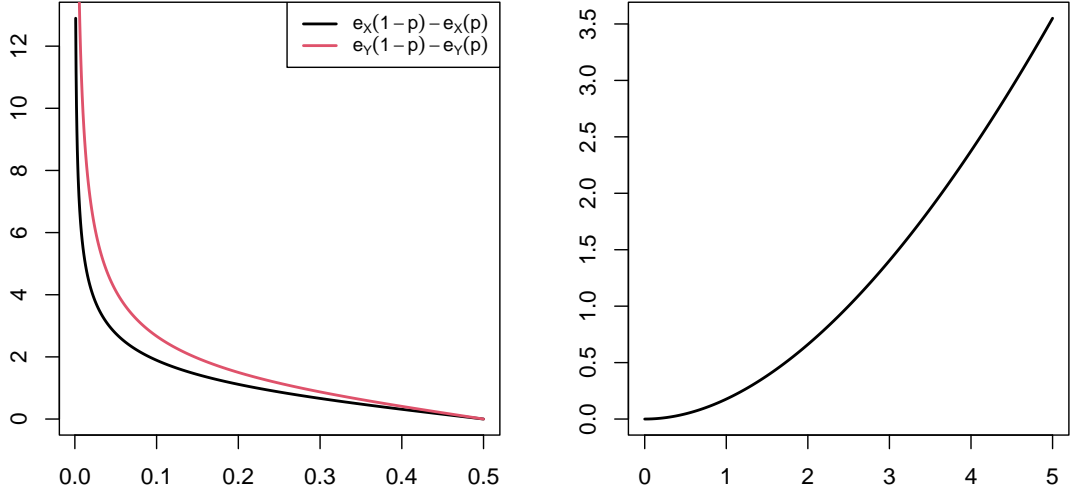}
\caption{\label{fig:lomax} Left panel: Interexpectile ranges of $X$ (in black) and $Y$ (in red) under Lomax distributions. Left panel: Plot of $e_Y(\breve{F}_X(x))-x$.}
}
\end{figure}

\end{enumerate}
\end{remark}

\section{Interexpectile ranges and their empirical counterparts} \label{sec-exp-scale}

It is clear that any functional of the form $E[\varphi(X-EX)]$, where $\varphi$ is a convex function, preserves the dilation order. Then, by Thm. \ref{theo-dil}, this also holds for the weak expectile dispersive order. Examples are the standard deviation
$\sigma(F)=\{E(X-EX)^2\}^{1/2}$ and the mean absolute deviation around the mean $E\lvert X-EX\rvert$.
In this section, however, we want to have a closer look on the interexpectile range (IER) $\mathcal{E}_{\alpha}=e_X(1-\alpha)-e_X(\alpha), \alpha \in (0,1/2)$.
These scale measures obviously preserve the $\leq_{we-disp}$ ordering, and, hence, also the dilation order; they have already appeared as a scaling factor in the definition of $s_2$.
Moreover, (implicit) interexpectile differences have been used to extract information about the risk-neutral distribution of a financial index by \cite{bellini2018b,bellini2020}.

Using the properties of expectiles (see, e.g., \citealt{bkmr}), $\mathcal{E}_{\alpha}$ obviously satisfies property D1. in the introduction.
Hence, it is a  measure of variability with respect to the dispersive order, the most fundamental variability ordering, but also with respect to the dilation order. The latter is an order with respect to the mean, whereas the first order is location-free.
Further properties of the theoretical IER can be found in \cite{bellini2018b}, Prop.\ 3.1.

The results in Sec.\ \ref{sec-exp-skew} show that the MAD arises quite naturally when dealing with expectile-based skewness orders. Our next result bounds $\mathcal{E}_{\alpha}$ in terms of the MAD.

\begin{theorem}
For $\alpha\in (0,1/2)$,
\begin{align*}
\frac{1-2\alpha}{1-\alpha}E\lvert X-\mu\rvert &< \mathcal{E}_{\alpha}
\, < \frac{1-2\alpha}{\alpha} E\lvert X-\mu\rvert.
\end{align*}
In particular, $\frac23 E\lvert X-\mu\rvert < \mathcal{E}_{1/4}
< 2 E\lvert X-\mu\rvert$.
\end{theorem}

\begin{proof}
For any $\tau\in (0,1)\setminus \{1/2\}$, the first order condition (\ref{exp_def_2}) can be rewritten  as
\begin{align}
E(X - e_X(\tau))_+ &= \frac{1-\tau}{1-2\tau}(\mu-e_X(\tau)),		\label{eqn:stopLossRewr}\\
E(X - e_X(\tau))_- &= \frac{\tau}{1-2\tau}(\mu-e_X(\tau)).		\label{eqn:negStopLossRewr}
\end{align}
Since $E(X - t)_+$ is strictly decreasing and $E(X - t)_-$ is strictly increasing,
and since $e_X(\alpha)<\mu <e_X(1-\alpha)$, we further obtain
\begin{align}
E(X-e_X(1-\alpha))_+ &< E(X-\mu)_+ < E(X-e_X(\alpha))_+, \label{eqn:stopLossMon2}\\
E(X-e_X(\alpha))_- &< E(X-\mu)_- < E(X - e_X(1 - \alpha))_-. \label{eqn:negStopLossMon2}
\end{align}
Adding the terms in (\ref{eqn:stopLossMon2}) and (\ref{eqn:negStopLossMon2}), and using equations (\ref{eqn:stopLossRewr}) and (\ref{eqn:negStopLossRewr}) then yields
\begin{align*}
\frac{\alpha}{1-2\alpha} \left(e_X(1-\alpha)-e_X(\alpha)\right)
&< E\lvert X-\mu\rvert
\, < \frac{1-\alpha}{1-2\alpha} \left(e_X(1-\alpha)-e_X(\alpha)\right),
\end{align*}
or
\begin{align*}
\frac{1-2\alpha}{1-\alpha}E\lvert X-\mu\rvert
&< e_X(1-\alpha)-e_X(\alpha)
\, < \frac{1-2\alpha}{\alpha} E\lvert X-\mu\rvert.
\end{align*}
\end{proof}

For any $\alpha\in(0,1/2)$, we define the population counterparts of $\mathcal{E}_{\alpha}$ by $\widehat{\mathcal{E}}_{\alpha,n} =\hat{e}_n(1-\alpha)-\hat{e}_n(\alpha)$.
For empirical expectiles, a multivariate central limit theorem as well as strong consistency holds \citep[see, e.g.,][]{holzmann}; from these results, a central limit theorem and strong consistency of $\widehat{\mathcal{E}}_{\alpha,n}$ can be derived. Using the notations
$\eta(\tau_1,\tau_2)= E[I_{\tau_1}(e_X(\tau_1),X) I_{\tau_2}(e_X(\tau_2),X)]$ for $\tau_1,\tau_2 \in (0,1)$
and $a_\alpha=\alpha+(1-2\alpha)F(e_X(\alpha))$,
the following holds true.

\begin{theorem}
Let $F_X$ be a cdf with $E\lvert X\rvert <\infty$, and $\alpha \in (0, 1/2)$.
\begin{enumerate}
\item[a)]
If $E X^2<\infty$ and $F_X$ does not have a point mass at $e_X(\alpha)$ or $e_X(1-\alpha)$, then
\begin{equation*}
\sqrt{n} \left( \widehat{\mathcal{E}}_{\alpha,n}-\mathcal{E}_\alpha \right) \VK \mathcal{N}(0, \sigma^2_\mathcal{E}(\alpha)),
\end{equation*}
where
\begin{align*}
\sigma^2_\mathcal{E}(\alpha) &=\,
\frac{\eta(\alpha,\alpha)}{a_\alpha^2}
- \frac{2\eta(\alpha,1-\alpha)}{a_\alpha\cdot a_{1-\alpha}}
+ \frac{\eta(1-\alpha,1-\alpha)}{a_{1-\alpha}^2}.
\end{align*}
	
\item[b)]
$\widehat{\mathcal{E}}_{\alpha,n}$ is a strongly consistent estimator of $\mathcal{E}_\alpha$, i.e.
\begin{equation*}
\widehat{\mathcal{E}}_{\alpha,n} \FSK \mathcal{E}_\alpha.
\end{equation*}
	\end{enumerate}
\end{theorem}

A natural competitor to $\mathcal{E}_\alpha$ is the interquantile range (IQR) $\mathcal{Q}_{\alpha} = q_X(1-\alpha)-q_X(\alpha), \alpha \in (0,1/2)$, where $q_X(\alpha)$ denotes the $\alpha$-quantile of the distribution of $X$.
By definition, the interquantile range preserves the dispersive order. However, it is not consistent with the dilation order, which may be seen as a disadvantage in specific applications \citep{bellini2020}.

A general comparison result between the IQR and IER is not possible: $\mathcal{Q}_{\alpha}$ may be smaller than $\mathcal{E}_\alpha$ for some $\alpha$, and larger for other ones. However, such a comparison is possible for symmetric log-concave distributions such as the normal, logistic, uniform or Laplace distribution.
This follows directly from Corollary 7 and the preceding results in \cite{arab}:

\begin{proposition} \label{prop-comp}
Let $F_X$ be a symmetric cdf with finite mean and log-concave density. Then,
\begin{align*}
 \mathcal{E}_\alpha &< \mathcal{Q}_{\alpha} \quad \text{for each } \alpha \in (0, 1/2).
\end{align*}
\end{proposition}

In the following, we compare the efficiency of the empirical IER as an estimator of the variability for specific distributions with other measures of dispersion, in particular with the IQR.
Writing  $\widehat{\mathcal{Q}}_{\alpha,n} =\hat{q}_n(1-\alpha)-\hat{q}_n(\alpha)$, where $\hat{q}_n(\alpha)$ is the sample quantile, one has
\begin{equation*}
\sqrt{n}(\widehat{\mathcal{Q}}_{\alpha,n}-\mathcal{Q}_\alpha)
\VK
\mathcal{N}\left(0,\sigma^2_\mathcal{Q}(\alpha)\right),
\end{equation*}
where
\begin{equation*}
\sigma^2_\mathcal{Q}(\alpha) =
\frac{\alpha(1-\alpha)}{f^2(q_X(1-\alpha))}
- \frac{2\alpha^2}{f(q_X(\alpha))f(q_X(1-\alpha))}
+\frac{\alpha(1-\alpha)}{f^2(q_X(\alpha))},
\end{equation*}
and $f$ denotes the density of $F_X$.

If $X\sim \mathcal{N}(\mu,\sigma^2)$, we obtain $\sigma^2_\mathcal{Q}(\alpha) =
2\alpha(1-2\alpha)/f^2(q_X(\alpha))$, and
$\mathcal{Q}(\alpha)=2\sigma\Phi^{-1}(1-\alpha)$. Therefore,
\begin{equation*}
\sqrt{n}\left(\frac{\widehat{\mathcal{Q}}_{\alpha,n}} {2\Phi^{-1}(1-\alpha)} - \sigma\right) \VK
\mathcal{N}\left(0,\tau^2_\mathcal{Q}(\alpha)\sigma^2\right),
\end{equation*}
where
\begin{equation*}
\tau^2_\mathcal{Q}(\alpha) =
\frac{\alpha(1-2\alpha)}{2\left(\Phi^{-1}(1-\alpha)\right)^2 \, \varphi^2\left(\Phi^{-1}(1-\alpha)\right)}.
\end{equation*}
On the other hand, the sample standard deviation
$\widehat{\sigma}_n=(1/(n-1)\sum_{i=1}^n(X_i-\bar{X}_n)^2)^{1/2}$
has asymptotic variance $(\mu_4-\sigma^4)/(4\sigma^2)$, where $\mu_4=E(X-\mu)^4$, which simplifies to $\sigma^2/2$ under normality.
Hence, the standardized asymptotic relative efficiency (ARE) is given by
\begin{equation*}
sARE(\widehat{\mathcal{Q}}_{\alpha,n},\widehat{\sigma}_n) = \frac{1/2}{\tau^2_\mathcal{Q}(\alpha)}.
\end{equation*}
We term $\tau^2_\mathcal{Q}(\alpha)$ the standardized asymptotic variance (standardized ASV).
Proceeding in the same way with the IER leads to the corresponding quantity
$\tau^2_\mathcal{E}(\alpha)=\sigma^2_\mathcal{E}(\alpha)/(2e_X(1-\alpha))^2$.
Figure \ref{fig:sARE} shows the standardized ASV's $\tau^2_\mathcal{Q}(\alpha)$ and $\tau^2_\mathcal{E}(\alpha)$ as functions of $\alpha$.

\begin{figure}
\centering{\includegraphics[width=0.95\textwidth]{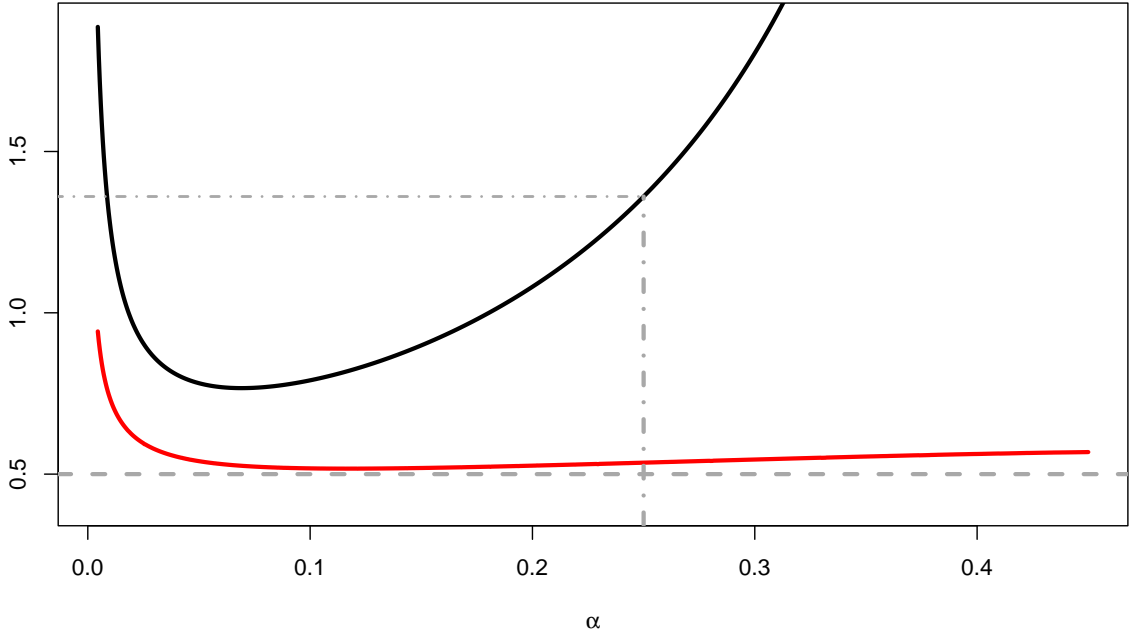}
\caption{\label{fig:sARE} Standardized asymptotic variances $\tau^2_\mathcal{Q}(\alpha)$ (in black) and $\tau^2_\mathcal{E}(\alpha)$ (in red) under normality.}
}
\end{figure}

For the interquartile range, i.e. the choice $\alpha=1/4$, one obtains the well known result $sARE(\widehat{\mathcal{Q}}_{1/4,n},\widehat{\sigma}_n)=0.368$; the standardized ARE takes the maximal value $0.652$ for $\alpha=0.0692$ \citep{fisher,david}. Whereas the latter estimator is more efficient, an advantage of using more central quantiles such as quartiles is their greater stability.

\smallskip
Proceeding to the IER, Fig.\ \ref{fig:sARE} shows that the standardized ASV's are generally smaller and quite stable over a large range of $\alpha$-values. For $\alpha=1/4$, we obtain $sARE(\widehat{\mathcal{E}}_{1/4,n},\widehat{\sigma}_n)=0.934$; the standardized ARE takes the maximal value $0.967$ for $\alpha=0.118$. Hence, the IER is a quite efficient scale estimator under normality compared to the standard deviation.

To analyze the behaviour of the IER under distributions with longer tails than the normal, we consider in the following Student's $t$-distribution with $\nu$ degrees of freedom, denoted by $t_{\nu}$, and having density function $c_{\nu}(1+x^2/\nu)^{-(\nu+1)/2}, x\in\mathbb{R}$, for $\nu>0$, where
$c_{\nu}=\Gamma((\nu+1)/2)/(\sqrt{\nu\pi}\Gamma(\nu/2))$.

Additionally to the dispersion measures used so far, we consider the MAD $d=E\lvert X-EX\rvert$, estimated by
$\widehat{d}_n=1/n \sum_{i=1}^n\lvert X_i-\bar{X}_n\rvert$,
and Gini's mean difference $g=E\lvert X-Y\rvert$, where $Y$ is a independent copy of $X$. The usual estimator of $g$ is the sample mean difference
$\widehat{g}_n=\frac{2}{n(n-1)}\sum_{i<j}\lvert X_i-X_j\rvert $.
The asymptotic variance of $\widehat{d}_n$ is given by
\begin{align*}
 ASV(\widehat{d}_n) &= Var\left(\lvert X-\mu\lvert + (2F(\mu)-1)X\right)
\end{align*}
\citep[Example 19.25]{vdw98}, which simplifies to
$ASV(\widehat{d}_n)=\sigma^2-d^2$ for symmetric distributions.
In this case, the ASV is the same as for the sample MAD around the median $1/n\sum_{i=1}^n\lvert X_i-q_X(1/2)\rvert $ \citep{gerst}. An explicit expression for the ASV of $\widehat{g}_n$ under the $t$-distribution can be found in \cite{gerst}, Table 3.

Similarly as above, given a scale measure $\delta$ and its estimator $\hat{\delta}$, we call the ratio $ASV(\hat{\delta})/\delta^2$ standardized asymptotic variance. Table \ref{table1} shows the standardized ASV of the different estimators for the $t$-distribution with various degrees of freedom.

\begin{table}
\centering
\begin{tabular}{lccccccccc}
  \toprule 
 & $\widehat{\sigma}$ & $\widehat{\mathcal{E}}_{0.15}$
 & $\widehat{\mathcal{E}}_{0.25}$ & $\widehat{\mathcal{E}}_{0.35}$
 & $\widehat{d}$ & $\widehat{g}$ & $\widehat{\mathcal{Q}}_{0.15}$
 & $\widehat{\mathcal{Q}}_{0.25}$ & $\widehat{\mathcal{Q}}_{0.35}$ \\
  \midrule 
$t_{3}$  & - & 2.020 & 1.651 & 1.517 & 1.467 & 1.907 & 1.330 & 1.613 & 2.726 \\
$t_{4}$  & - & 1.198 & 1.057 & 1.012 & 1.000 & 1.165 & 1.200 & 1.540 & 2.685 \\
$t_{5}$  & 2.000 & 0.949 & 0.871 & 0.851 & 0.851 & 0.932 & 1.129 & 1.500 & 2.661 \\
$t_{6}$  & 1.250 & 0.832 & 0.782 & 0.774 & 0.778 & 0.820 & 1.085 & 1.475 & 2.646 \\
$t_{7}$  & 1.000 & 0.765 & 0.730 & 0.728 & 0.735 & 0.754 & 1.055 & 1.457 & 2.636 \\
$t_{8}$  & 0.875 & 0.721 & 0.696 & 0.698 & 0.707 & 0.712 & 1.034 & 1.444 & 2.628 \\
$t_{9}$  & 0.800 & 0.690 & 0.672 & 0.677 & 0.687 & 0.682 & 1.017 & 1.434 & 2.622 \\
$t_{10}$ & 0.750 & 0.667 & 0.655 & 0.661 & 0.672 & 0.659 & 1.004 & 1.426 & 2.617 \\
$t_{12}$ & 0.688 & 0.636 & 0.630 & 0.639 & 0.651 & 0.629 & 0.986 & 1.415 & 2.611 \\
$t_{15}$ & 0.636 & 0.608 & 0.608 & 0.619 & 0.632 & 0.601 & 0.967 & 1.403 & 2.604 \\
$t_{20}$ & 0.594 & 0.583 & 0.587 & 0.601 & 0.615 & 0.575 & 0.949 & 1.392 & 2.597 \\
$t_{30}$ & 0.558 & 0.559 & 0.569 & 0.584 & 0.599 & 0.552 & 0.932 & 1.382 & 2.590 \\
$t_{40}$ & 0.542 & 0.549 & 0.560 & 0.576 & 0.592 & 0.541 & 0.924 & 1.376 & 2.587 \\
$t_{50}$ & 0.533 & 0.543 & 0.555 & 0.572 & 0.587 & 0.535 & 0.919 & 1.373 & 2.585 \\
$t_{100}$& 0.516 & 0.531 & 0.545 & 0.563 & 0.579 & 0.523 & 0.909 & 1.367 & 2.581 \\
  \bottomrule 
\end{tabular}
\caption{\label{table1} Standardized ASV of different scale estimators for the $t_{\nu}$-distribution.}
\end{table}

Concerning the comparison between the interquantile and interexpectile range, we observe a similar behavior as for the normal distribution. Whereas the standardized ASV's vary strongly with $\alpha$ for the IQR, they are quite stable for the IER. To allow for a better comparison of the relative efficiencies, Table \ref{table2} shows the minimum of the standardized ASV's in each line of Table \ref{table1}, divided by the standardized ASV's. Hence, for each distribution, the table shows the sARE with respect to the most efficient estimator. Whereas $\widehat{\mathcal{Q}}_{0.15}$ and $\widehat{d}$ are most efficient for $\nu=3$ and $\nu=4,5$, respectively, the IER gets in the lead for degrees of freedom between 6 and 10. For higher degrees of freedom, Gini's mean difference and the standard deviation are most efficient.

\begin{table}
\centering
\begin{tabular}{lccccccccc}
  \toprule 
 & $\widehat{\sigma}$ & $\widehat{\mathcal{E}}_{0.15}$
 & $\widehat{\mathcal{E}}_{0.25}$ & $\widehat{\mathcal{E}}_{0.35}$
 & $\widehat{d}$ & $\widehat{g}$ & $\widehat{\mathcal{Q}}_{0.15}$
 & $\widehat{\mathcal{Q}}_{0.25}$ & $\widehat{\mathcal{Q}}_{0.35}$ \\
  \midrule 
$t_{3}$  &  -    & 0.659 & 0.806 & 0.877 & 0.907 & 0.698 & \textbf{1.000} & 0.825 & 0.488 \\
$t_{4}$  &  -    & 0.835 & 0.946 & 0.988 & \textbf{1.000} & 0.859 & 0.834 & 0.649 & 0.372 \\
$t_{5}$  & 0.425 & 0.896 & 0.976 & 0.999 & \textbf{1.000} & 0.913 & 0.753 & 0.567 & 0.320 \\
$t_{6}$  & 0.619 & 0.930 & 0.989 & \textbf{1.000} & 0.995 & 0.944 & 0.713 & 0.525 & 0.292 \\
$t_{7}$  & 0.728 & 0.952 & 0.997 & \textbf{1.000} & 0.991 & 0.965 & 0.690 & 0.500 & 0.276 \\
$t_{8}$  & 0.796 & 0.966 & \textbf{1.000} & 0.998 & 0.985 & 0.978 & 0.674 & 0.482 & 0.265 \\
$t_{9}$  & 0.840 & 0.974 & \textbf{1.000} & 0.993 & 0.978 & 0.986 & 0.660 & 0.469 & 0.256 \\
$t_{10}$ & 0.873 & 0.981 & \textbf{1.000} & 0.990 & 0.974 & 0.993 & 0.652 & 0.459 & 0.250 \\
$t_{12}$ & 0.915 & 0.988 & 0.998 & 0.984 & 0.966 & \textbf{1.000} & 0.638 & 0.444 & 0.241 \\
$t_{15}$ & 0.944 & 0.988 & 0.989 & 0.970 & 0.950 & \textbf{1.000} & 0.621 & 0.428 & 0.231 \\
$t_{20}$ & 0.969 & 0.987 & 0.980 & 0.957 & 0.936 & \textbf{1.000} & 0.606 & 0.413 & 0.222 \\
$t_{30}$ & 0.990 & 0.987 & 0.971 & 0.945 & 0.922 & \textbf{1.000} & 0.592 & 0.400 & 0.213 \\
$t_{40}$ & 0.999 & 0.986 & 0.967 & 0.939 & 0.915 & \textbf{1.000} & 0.586 & 0.393 & 0.209 \\
$t_{50}$ & \textbf{1.000} & 0.982 & 0.960 & 0.931 & 0.907 & 0.996 & 0.580 & 0.388 & 0.206 \\
$t_{100}$& \textbf{1.000} & 0.972 & 0.946 & 0.916 & 0.891 & 0.986 & 0.567 & 0.377 & 0.200 \\
  \bottomrule 
\end{tabular}
\caption{\label{table2}
sARE of different scale estimators with respect to the most efficient estimator for the $t_{\nu}$-distribution.}
\end{table}

As a final example, we use the normal inverse Gaussian (NIG) distribution \citep{barndorff}, which allows for skewed and heavy-tailed distributions; further, all moments exist and have simple explicit expressions.
Table \ref{table3} shows the standardized ASV's of the different estimators for the NIG distribution with shape parameters $\alpha$ and $\beta$. For $\beta=0$, the distribution is symmetric. The two remaining parameters are chosen such that $EX=0$ and $Var(X)=1$; hence, the third and fourth moment in columns 3-4 corresponds to the moment skewness and kurtosis. Note that, for $\beta=0$ and $\alpha\to\infty$, the distribution converges to the standard normal.

\begin{table}
\centering
\begin{tabular}{rccc|ccccc}
  \toprule
   $\alpha$ & $\beta$ & $m_3$ & $m_4$ & $\widehat{\sigma}$ & $\widehat{\mathcal{E}}_{0.25}$
 & $\widehat{d}$ & $\widehat{g}$ & $\widehat{\mathcal{Q}}_{0.25}$ \\
 \midrule
      10.0 & 0.0 & 0.00 & 3.03 & 0.507 & 0.540 & 0.575 & 0.517 & 1.364 \\
      10.0 & 8.0 & 0.67 & 3.82 & 0.706 & 0.654 & 0.679 & 0.620 & 1.009 \\
      10.0 & 9.0 & 1.42 & 6.52 & 1.381 & 1.017 & 1.016 & 0.938 & 0.713 \\
      2.0 & 0.0 & 0.00 & 3.75 & 0.688 & 0.642 & 0.662 & 0.640 & 1.433 \\
      2.0 & 1.0 & 1.00 & 5.67 & 1.167 & 0.882 & 0.881 & 0.848 & 0.990 \\
      2.0 & 1.5 & 2.57 & 15.73 & 3.684 & 1.979 & 1.891 & 1.735 & 0.655 \\
      1.0 & 0.0 & 0.00 & 6.00 & 1.250 & 0.910 & 0.884 & 0.947 & 1.602 \\
      1.0 & 0.5 & 2.00 & 13.67 & 3.167 & 1.654 & 1.564 & 1.531 & 1.020 \\
      1.0 & 0.8 & 6.67 & 85.41 & 21.102 & 7.112 & 6.607 & 5.379 & 0.877 \\
  \bottomrule
\end{tabular}
\caption{\label{table3} Standardized ASV of different scale estimators for the NIG-distribution.}
\end{table}

From Table \ref{table3} we see that $\widehat{\mathcal{E}}_{0.25}, \widehat{d}$ and $\widehat{g}$ have comparable standardized ASV's. They are all quite efficient compared to $\widehat{\sigma}$ for distributions near to the normal. They are considerably more efficient than $\widehat{\sigma}$ for skewed and long-tailed distributions as the NIG distribution with $\alpha=1, \beta=0.5$, and compare well with the IQR in this case. All in all, these measures seem to be a reasonable compromise between the standard deviation on the one hand, and the interquartile range on the other.

\subsubsection*{Acknowledgments}
We thank the two anonymous reviewers for their constructive and helpful comments. In particular, we thank one of the reviewers for pointing to the result stated as Prop. \ref{prop-comp}.

\end{document}